\documentclass{amsart}
\usepackage{amsmath}
\usepackage{amssymb}
\usepackage{amscd}
\usepackage[all]{xy}
\usepackage{color}
\usepackage{url}

\newtheorem{theorem}{Theorem}
\newtheorem{lemma}[theorem]{Lemma}
\newtheorem{corollary}[theorem]{Corollary}
\newtheorem{proposition}[theorem]{Proposition}
\newtheorem{definition}[theorem]{Definition}

\newtheorem{remark}[theorem]{Remark}

\newcommand{\Z}{\mathbb{Z}}
\newcommand{\N}{\mathbb{N}}

\DeclareMathOperator{\Pic}{Pic}

\begin{document}

\title[ ]
{Semigroup rings as weakly Krull domains}
\author[G.W. Chang]{Gyu Whan Chang}
\author[V. Fadinger]{Victor Fadinger}
\author[D. Windisch]{Daniel Windisch}

\address{Department of Mathematics Education, Incheon Nation University, Incheon 22012, South Korea}
\email{whan@inu.ac.kr}
\address{University of Graz, NAWI Graz \\
Institute for Mathematics and Scientific Computing \\
Heinrichstra{\ss}e 36\\
8010 Graz, Austria}
\email{victor.fadinger@uni-graz.at}
\address{Graz University of Technology, NAWI Graz\\
Institute for Analysis and Number Theory\\
Kopernikusgasse 24/II\\
8010 Graz, Austria}
\email{dwindisch@math.tugraz.at}

\thanks{$2010$ Mathematics Subject Classification : 13A15, 13F05, 20M12}
\thanks{Words and phrases : Semigroup ring, weakly Krull domain, finite $t$-character,  system of sets of lengths}

\date{\today}

\begin{abstract}
Let $D$ be an integral domain and $\Gamma$ be a torsion-free commutative cancellative (additive) semigroup with identity element and
quotient group $G$. In this paper, we show that if char$(D)=0$ (resp., char$(D)=p>0$), then
$D[\Gamma]$ is a weakly Krull domain if and only if $D$ is a weakly Krull UMT-domain,
$\Gamma$ is a weakly Krull UMT-monoid, and $G$ is of type $(0,0,0, \dots )$ (resp., type $(0,0,0, \dots )$ except $p$).
Moreover, we give arithmetical applications of this result.
\end{abstract}
\maketitle

\section{Introduction}

Let $D$ be an integral domain and $X^1(D)$ be the set of height-one
prime ideals of $D$. We say that $D$ is a Krull domain if $D$ satisfies the following three properties:

\vspace{.2cm}

(i) $D = \bigcap_{P \in X^1(D)}D_P$,

(ii) each nonzero nonunit of $D$ is contained in only finitely many height-one prime ideals
of $D$, and

(iii) $D_P$ is a principal ideal domain (PID) for all $P \in X^1(D)$.

\vspace{.2cm}
\noindent
Krull domains include UFDs and Dedekind domains. However, many well-studied rings are close to being Krull by satisfying (i) and (ii),
but property (iii) fails, e.g. non-principal orders in number fields and $\mathbb{Q}[X^2, X^3]$ for an indeterminate $X$
over the field $\mathbb{Q}$ of rational numbers. An integral domain satisfying (i) and (ii) is called a weakly Krull domain.
Hence, Krull domains and one-dimensional noetherian domains are weakly Krull, but the backwards implications need not hold true.
The notion of weakly Krull domains was first introduced and
studied by Anderson, Mott and Zafrullah \cite{amz92}.
A weakly factorial domain (WFD) is an integral domain whose nonzero elements
can be written as finite products of primary elements. Then
UFDs are WFDs, and $D$ is a WFD if and only if $D$ is a weakly Krull domain
and each $t$-invertible $t$-ideal of $D$ is principal \cite[Theorem]{az90}.

Let $\Gamma$ be a monoid, i.e., a commutative cancellative (additive) semigroup with identity element
and $D[\Gamma]$ be the semigroup ring of $\Gamma$ over $D$.
Then $\Gamma$ has a quotient group \cite[Theorem 1.2]{g84}, and $D[\Gamma]$
 is an integral domain if and only if $\Gamma$ is torsion-free \cite[Theorem 8.1]{g84}.
It is well known that $D[\Gamma]$ is a Krull domain (resp., UFD)
if and only if $D$ is a Krull domain (resp., UFD),
$\Gamma$ is a Krull monoid (resp., factorial monoid), and $\langle \Gamma \rangle$,
the quotient group of $\Gamma$, satisfies the ascending chain condition
on its cyclic subgroups \cite[Theorem 15.6]{g84} (resp., \cite[Theorem 14.16]{g84}).
In \cite{c09}, Chang characterized when $D[\Gamma]$ is a WFD under the
assumption that $\langle \Gamma \rangle$ satisfies the ascending chain condition
on its cyclic subgroups. Then, he asked when $D[\Gamma]$ becomes a weakly Krull domain \cite[Question 14]{c09}.
Furthermore, in \cite{co19}, Chang and Oh completely characterized the weakly factorial property of $D[\Gamma]$.
Recently, in \cite{fw21}, Fadinger and Windisch gave a partial answer to Chang's question using the concept of weakly Krull monoids.
These are defined analogously  to weakly Krull domains by properties (i) and (ii) above,
which were introduced and
characterized by Halter-Koch in \cite{hk95}. There he
proved that $D$ is weakly Krull if and only if its multiplicative monoid $D\setminus \{0\}$ is a weakly Krull monoid.

Now, in this paper, we give a complete characterization of weaky Krull
semigroup rings $D[\Gamma]$. Precisely, we show that
if char$(D)=0$ (resp., char$(D)=p>0$), then
$D[\Gamma]$ is a weakly Krull domain if and only if $D$ is a weakly Krull UMT-domain,
$\Gamma$ is a weakly Krull UMT-monoid, and the quotient group
of $\Gamma$ is of type $(0,0,0, \dots )$ (resp., type $(0,0,0, \dots )$ except $p$).
As a corollary, we recover Matsuda's results \cite{m77, m82} that if char$(D)=0$ (resp., char$(D) = p >0$),
then $D[\Gamma]$ is a generalized Krull domain if and only if $D$ is a generalized Krull domain,
$\Gamma$ is a generalized Krull monoid, and $\langle \Gamma \rangle$ is of type $(0,0,0, \dots )$
(resp., type $(0,0,0, \dots )$ except $p$).

Moreover, in the final section we use the main result in order to obtain arithmetical statements on weakly Krull semigroup rings.
For instance, we provide a large class of weakly Krull numerical semigroup rings that have full systems of sets of lengths.
Also for a certain class of affine semigroup rings we prove a result about the connection of its class group and its system of sets of lengths.
Thereby, we are the first to give a fairly broad but sufficiently concrete class of non-local weakly Krull domains
that are not Krull, but whose arithmetic is still accessible.

\section{Definitions related to the $t$-operation and monoids}

Let $D$ be an integral domain with quotient field $K$, $\Gamma$ be a torsion-free
monoid with
quotient group $\langle \Gamma \rangle$; so $D[\Gamma]$ is an integral domain,
$\bar{D}$ be the integral closure of $D$ in $K$,
and $\bar{\Gamma}$ be the integral closure (i.e., root closure) of $\Gamma$ in $\langle \Gamma \rangle$.
 If we say that $D$ is local, we do not impose that $D$ is noetherian.

\subsection{The $t$-operation on integral domains}
Let $F(D)$ be the set of nonzero fractional ideals of $D$. For $I \in F(D)$,
let $I^{-1} = \{ x \in K \mid xI \subseteq D\}$. It is easy to see that $I^{-1} \in F(D)$.
Hence, the $v$- and $t$-operations are well-defined as follows:
\begin{enumerate}
\item $I_v = (I^{-1})^{-1}$ and
\item $I_t = \bigcup \{ J_v \mid J \text{ is a finitely generated subideal of } I\}$.
\end{enumerate}
Let $* = v$ or $*=t$. Then, for any nonzero $a \in K$ and $I, J \in F(D)$,
(i) $aI_* = (aI)_*$,
(ii) $I \subseteq I_*$; $I \subseteq J$ implies $I_* \subseteq J_*$,
(iii) $(I_*)_* = I_*$, and (iv) $(IJ)_* = (IJ_*)_*$.
An $I \in F(D)$ is called a {\it $*$-ideal} if $I_* = I$.

A $t$-ideal is a
{\it maximal $t$-ideal} of $D$ if it is maximal among proper integral $t$-ideals of $D$.
It is easy to see that each maximal $t$-ideal is a prime ideal, each $t$-ideal is contained in a maximal $t$-ideal,
a prime ideal minimal over a $t$-ideal is a $t$-ideal, each nonzero principal ideal is a $v$-ideal,
each $v$-ideal is a $t$-ideal, $I \subseteq I_t \subseteq I_v$ for all $I \in F(D)$, and $I_t = I_v$ if $I$ is finitely generated.
Let Max$(D)$ (resp., $t$-Max$(D)$) be the set of maximal ideals (resp., maximal $t$-ideals) of $D$.
It is easy to see that $D = \bigcap_{M \in \text{Max}(D)}D_M = \bigcap_{P \in t\text{-Max}(D)}D_P$.
By $t$-dim$(D)=1$, we mean that $D$ is not a field and each prime $t$-ideal of $D$ is a maximal $t$-ideal,
and in this case, $X^1(D) = t$-Max$(D)$.
It is well known that if $D$ is not a field, then
$D$ is a weakly Krull domain if and only if $t$-dim$(D)=1$ and
$D$ is of finite $t$-character (i.e., each nonzero nonunit of $D$
is contained in only finitely many maximal $t$-ideals) \cite[Lemma 2.1]{amz92}.

An $I \in F(D)$ is said to be invertible (resp., $t$-invertible, $v$-invertible) if $II^{-1} =D$ (resp., $(II^{-1})_t = D$, $(II^{-1})_v = D$).
Let $T(D)$ be the abelian group of $t$-invertible fractional $t$-ideals of $D$
under the $t$-multiplication $I*J = (IJ)_t$
and $Prin(D)$ be the set of nonzero principal fractional ideals of $D$.
Then $Prin(D)$ is a subgroup of $T(D)$, and
$Cl_t(D) = T(D)/Prin(D)$ denotes the factor group of $T(D)$ modulo $Prin(D)$.
We denote the group of all $v$-invertible fractional $v$-ideals by $F_v(D)^{\times}$ and the monoid of all $v$-invertible integral $v$-ideals
by $\mathcal I_v^*(D)$ where multiplication is defined via $I\cdot_v J=(I\cdot J)_v$.
Then $F_v(D)^\times$ is the quotient group of $\mathcal I_v^*(D)$. By $\mathcal C_v(D)$ we denote the quotient of $F_v(D)^\times$ modulo
$Prin(D)$ and call it the ($v$-)\textit{class group} of $D$. If we denote the set of all non-zero (integral) principal ideals of $D$ by $\mathcal H(D)$,
then the embedding $\mathcal H(D)\to \mathcal I_v^*(D)$ is a cofinal divisor homomorphism and $\mathcal I_v^*(D)/\mathcal H(D)=\mathcal C_v(D)$
(for more in this direction see \cite[Chapter 2.10]{GHK}).
It is easy to see that  a $t$-invertible $t$-ideal is a $v$-invertible $v$-ideal.
Hence, $Prin(D) \subseteq T(D) \subseteq F_v(D)^{\times}$, and thus $Cl_t(D)$ is a subgroup of $\mathcal C_v(D)$,
and equality holds if $D$ is a Mori domain (e.g., Krull domain). An integral
domain is a Mori domain if it satisfies the ascending chain condition on its integral $v$-ideals.

\subsection{Monoids.}
Let $H$ be a monoid with quotient group $\langle H\rangle$.
As in the case of integral domains, we can define the
$v$-operation, the $t$-operation, $t$-Max$(H)$, $t$-invertibility, the class groups, and the Mori monoid.
The reader is referred to \cite{h98} for more on the $v$- and $t$-operation on monoids.
A monoid $H$ is a Krull monoid if and only if $H$ is a completely integrally closed Mori monoid,
if and only if each ideal of $H$ is $t$-invertible \cite[Theorem 22.8]{h98}.

Let $G$ be a torsion-free abelian group.
We say that $G$ is {\em of type} $(0,0,0, \ldots)$
if $G$ satisfies the ascending chain condition on its cyclic subgroups (equivalently, for each nonzero element $g \in G$,
there exists a largest positive integer $n_g$ such that $n_g x = g$ is solvable in $G$) \cite[Theorem 14.10]{g84}.
For a prime number $p$, $G$ is said to be of type $(0,0,0, \ldots)$ except $p$
if $G$ satisfies the following two conditions; for each nonzero element $g \in G$,  (i)
an infinite number of prime numbers do not divide $g$
and (ii) for each prime number $q \neq p$, $q^n$ does not divide $g$ for some positive integer $n$.
Clearly, a torsion-free abelian group of  type $(0,0,0, \ldots)$ is of type $(0,0,0, \ldots)$ except $p$ for
all prime numbers $p$, but not vice versa. For example, let $G = \bigcup_{n=1}^{\infty} (1/p^n ) \mathbb{Z}$ for a prime number $p$,
then $G$ is of type $(0,0,0, \ldots)$ except $p$ but not of type $(0,0,0, \ldots)$.
The notion of  type $(0,0,0, \ldots)$ except $p$ was introduced by Matsuda \cite{m77, m82}
in order to study when $D[G]$ is a generalized Krull domain for an integral domain $D$ with char$(D) = p$.

\begin{remark}
{\em (1) Let $G$ be a nonzero torsion-free abelian group and $p >0$ be a prime number.
In \cite{co19}, Chang and Oh say that $G$ is of type $(0,0,0, \ldots)$ except $p$
if $G$ satisfies the following two conditions for each nonzero element $g \in G$;

\vspace{.1cm}

(i$'$) the number of prime numbers dividing $g$ is finite and

(ii) for each prime number $q \neq p$, $q^n$ does not divide $g$ for some integer $n \geq 1$.

\vspace{.1cm}
\noindent
Then, in \cite[Theorem 4.2]{co19}, they prove that if $D$ is an integral domain with char$(D) = p >0$,
then $D[G]$ is of finite $t$-character if and only if
$D$ is of finite $t$-character and $G$ is of type $(0,0,0, \ldots)$ except $p$.
But, in order to prove \cite[Theorem 4.2]{co19}, they actually used Matsuda's original definition
of type $(0,0,0, \ldots)$ except $p$.  Thus, there is no problem when we cite the results of \cite{co19}.

\vspace{.1cm}
(2) Let $S$ be an infinite set of prime numbers such that there are also infinitely many prime numbers in $\mathbb{Z} \setminus S$
and $p \not\in S$.
Let $m$ be a positive integer and $$G = \{\frac{a}{p_1^{e_1} \cdots p_k^{e_k}p^n} \mid a \in \mathbb{Z},
p_i \in S, 0 \leq e_i \leq m \text{ for } i =1, \dots , k, \text{ and } n \geq 1\}.$$
Then $G$ is a torsion-free abelian group under the usual addition.
Moreover, $G$ is of type $(0,0,0, \ldots)$ except $p$, but $G$ does not satisfy (i$'$)
(for example, $1 \in G$ and each $q \in S$ divides $1$
because $1 = q \cdot \frac{1}{q}$ and $\frac{1}{q} \in G$). Thus, (i$'$) together with (ii)
is stronger than (i) and (ii).}
\end{remark}

\subsection{Semigroup rings}
Let $\Gamma$ be a torsion-free monoid.
It is well known that $\Gamma$
admits a total order $<$ compatible with its monoid operation \cite[Corollary 3.4]{g84}.
Hence each $f \in D[\Gamma]$ is uniquely expressible in the form
$f = a_1X^{\alpha_1} + a_2X^{\alpha_2} + \cdots + a_kX^{\alpha_k}$,
where $a_i \in D$ and $\alpha_j \in \Gamma$ with $\alpha_1 < \alpha_2 < \cdots < \alpha_k$.
For an ideal $I$ (resp., $J$) of $D$ (resp., $\Gamma$),
let $I[J] = \{a_1X^{\alpha_1} + a_2X^{\alpha_2} + \cdots + a_kX^{\alpha_k} \in D[\Gamma] \mid a_i \in I$ and $\alpha_j \in J\}$.
Then $I[J]$ is an ideal of $D[\Gamma]$ \cite[Lemma 2.3]{eik02},
and $I[J]$ is a prime ideal if and only if either
$I$ is a prime ideal of $D$ and $J = \Gamma$ or $I = D$ and $J$ is a prime ideal of $\Gamma$ (cf. \cite[Corollary 8.2]{g84}
and the proof of \cite[Lemma 3.1]{fw21}).

\subsection{UMT-domains and UMT-monoids}
Let $X$ be an indeterminate over $D$ and $D[X]$ be the polynomial ring over $D$.
A nonzero prime ideal $Q$ of $D[X]$ is called an upper to zero in $D[X]$
if $Q \cap D = (0)$. So $Q$ is an upper to zero in $D[X]$ if and only if
$Q = fK[X] \cap D[X]$ for some irreducible polynomial $f \in K[X]$.
Following \cite{hz89}, we say that $D$ is a UMT-domain if each upper to zero in $D[X]$ is
a maximal $t$-ideal of $D[X]$. It is known that $D$ is a UMT-domain
if and only if $\bar{D}_P$ is a Pr\"ufer domain for all $P \in t$-Max$(D)$ \cite[Theorem 1.5]{fgh98}.

In \cite[Theorem 17]{cs18}, it was shown that $D[\Gamma]$ is a UMT-domain if and only if
$D$ is a UMT-domain and $\bar{\Gamma}_S$ is a valuation monoid for all maximal $t$-ideals
$S$ of $\Gamma$. Hence, the
following is a natural generalization of UMT-domains to monoids.

\begin{definition}
{\em Let $\Gamma$ be a torsion-free monoid with quotient group $G$
and $\bar{\Gamma}$ be the integral closure (i.e., root closure) of $\Gamma$ in $G$.
We say that $\Gamma$ is a {\em UMT-monoid} if $\bar{\Gamma}_S$ is a valuation monoid for all maximal $t$-ideals
$S$ of $\Gamma$.}
\end{definition}

A Pr\"ufer $v$-multiplication domain (P$v$MD) is an integral domain
whose nonzero finitely generated ideals are $t$-invertible.
Then $D$ is a P$v$MD if and only if $D_P$ is a valuation domain for
all maximal $t$-ideals $P$ of $D$ \cite[Theorem 5]{g67},
if and only if $D$ is an integrally closed UMT-domain \cite[Proposition 3.2]{hz89}.
Now, a monoid $\Gamma$ is called a Pr\"ufer $v$-multiplication monoid (P$v$MM)
if every finitely generated ideal of $\Gamma$ is $t$-invertible.
It is known that $\Gamma$ is a P$v$MM if and only if $\Gamma_S$ is a valuation monoid for all
$S \in t$-Max$(\Gamma)$ \cite[Theorem 17.2]{h98}, hence we have

\begin{proposition} \label{prop6}
Let $\Gamma$ be an integrally closed torsion-free monoid.
Then $\Gamma$ is a UMT-monoid if and only if $\Gamma$ is a P$v$MM.
\end{proposition}

By Proposition \ref{prop6}, UMT-monoids include valuation monoids, P$v$MMs, Krull monoids,
and monoids of $t$-dimension one whose integral closure is a P$v$MM.

\section{Weakly Krull domains; Preliminary Results}

Weakly Krull domains are of finite $t$-character. The class of integral domains of finite $t$-character
includes Krull domains and Noetherian domains. In this section, we recall a couple of
results on weakly Krull domains some of which are already known.
The following proposition seems to be of this kind, but we could not find a proper reference, so we provide a full proof.

\begin{proposition} \label{prop1}
Let $D$ be an integral domain with quotient field $K$, and assume that $D \neq K$.
\begin{enumerate}
\item Let $D$ be a weakly Krull domain and $S$ be a multiplicative subset of $D$. Then $D_S$ is also a weakly Krull domain.
\item Let $\{S_{\lambda}\}$ be a set of multiplicative subsets of $D$ such that $D = \bigcap_{\lambda}D_{S_{\lambda}}$ and
$D_{S_{\lambda}}$ is a weakly Krull domain for all $\lambda$.
If $D = \bigcap_{\lambda}D_{S_{\lambda}}$ has finite character, then
$D$ is a weakly Krull domain.
\end{enumerate}
\end{proposition}

\begin{proof}
(1) \cite[Proposition 4.7]{ahz93}.

(2)  Let
$X^1(D_{S_{\lambda}})$ be the set of height-one prime ideals of $D_{S_{\lambda}}$.
Then $$D= \bigcap_{\lambda}(\bigcap_{P \in X^1(D_{S_{\lambda}})}(D_{S_{\lambda}})_P)$$
and this intersection has finite character. Now, for $P \in X^1(D_{S_{\lambda}})$,
let $Q = P \cap D$. Then $D_Q = (D_{S_{\lambda}})_P$, and hence
$D = \bigcap_{Q \in T} D_Q$ for some $T \subseteq X^1(D)$ and this intersection has finite character.
Next, let $Q'$ be a height-one prime ideal of $D$.
Then $D_{Q'} = \bigcap_{Q \in T} (D_Q)_{D \setminus Q'}$ because the intersection has finite character.
Since $D_Q$ is a one-dimensional local domain, $(D_Q)_{D \setminus Q'} = K$
or $(D_Q)_{D \setminus Q'} = D_Q$. Also, since $D_{Q'}$ is a one-dimensional local domain,
$D_{Q'} = D_Q$ for some $Q \in T$. Thus, $T = X^1(D)$, so $D$ is a weakly Krull domain.
\end{proof}

Let $D$ be an integral domain with quotient field $K$ and $\Gamma$ be a torsion-free monoid with
with quotient group $G$.
For any $f = a_1X^{\alpha_1} + a_2X^{\alpha_2} + \cdots + a_kX^{\alpha_k} \in D[\Gamma]$ with $\alpha_1< \cdots < \alpha_k$,
let $C(f)$ be the ideal of $D[\Gamma]$ generated by $a_1X^{\alpha_1}, a_2X^{\alpha_2}, \dots , a_kX^{\alpha_k}$
and $c(f)$ be the ideal of $D$ generated by $a_1, \dots , a_k$, so $C(f) \subseteq c(f)D[\Gamma]$.
For convenience, we always assume that $f \neq 0$ when we study the $v$-closure $C(f)_v$ of $C(f)$.
Let $N(H) = \{f \in D[\Gamma] \mid C(f)_v= D[\Gamma]\}$ and $H = \{aX^{\alpha} \mid 0 \neq a \in D$ and $\alpha \in \Gamma\}$.
It is easy to see that $H$ and $N(H)$ are
saturated multiplicative subsets of $D[\Gamma]$, $D[\Gamma]_H = K[G]$,
and $D[\Gamma] = D[\Gamma]_{N(H)} \cap D[\Gamma]_H$.

\begin{lemma} \label{lemma1}
Let $D$ be an integral domain with quotient field $K$, $\Gamma$ be a torsion-free monoid with
quotient group $G$, and $N(H) = \{f \in D[\Gamma] \mid C(f)_v= D[\Gamma]\}$.
\begin{enumerate}
\item \begin{eqnarray*}
t\text{-Max}(D[\Gamma]) &=& \{P[\Gamma] \mid P \in t\text{-Max}(D)\}\\
                  &\cup& \{D[S] \mid S \in t\text{-Max}(\Gamma)\}\\
                  &\cup& \{Q \in t\text{-Max}(D[\Gamma]) \mid QK[G] \subsetneq K[G]\}.
\end{eqnarray*}
\item Max$(D[\Gamma]_{N(H)}) = \{P[\Gamma]_{N(H)} \mid P \in t\text{-Max}(D)\} \cup
\{D[S]_{N(H)} \mid S \in t\text{-Max}(\Gamma)\}$.
\item $t$-Max$(D[\Gamma]_{N(H)}) =$ Max$(D[\Gamma]_{N(H)})$.
\end{enumerate}
\end{lemma}

\begin{proof}
(1) \cite[Corollary 1.3]{ac05} and \cite[Corollary 2.4]{eik02}. (2) \cite[Proposition 1.4 and Example 1.6]{ac13}.
(3) \cite[Example 1.6 and Proposition 1.8]{ac13}.
\end{proof}

\begin{proposition} \label{prop3}
Let $D$ be an integral domain with quotient field $K$, $\Gamma$ be a torsion-free monoid with
quotient group $G$, and $N(H) = \{f \in D[\Gamma] \mid C(f)_v= D[\Gamma]\}$. Then
$D[\Gamma]$ is a weakly Krull domain if and only if
$D[\Gamma]_{N(H)}$ is a one-dimensional weakly Krull domain and $K[G]$ is a weakly Krull domain.
\end{proposition}

\begin{proof}
Let $H = \{aX^{\alpha} \mid 0 \neq a \in D$ and $\alpha \in \Gamma\}$ and $N = N(H)$.
Then $K[G] = D[\Gamma]_H$ and $D[\Gamma] = D[\Gamma]_N \cap D[\Gamma]_H$, and hence
$D[\Gamma]$ is a weakly Krull domain if and only if
both $D[\Gamma]_N$ and $K[G]$ are weakly Krull domains by Proposition \ref{prop1}.

Now, by Lemma \ref{lemma1}(3),
$t$-Max$(D[\Gamma]_{N(H)}) =$ Max$(D[\Gamma]_{N(H)})$.
Note that if $D[\Gamma]$ is a weakly Krull domain, then
$t$-dim$(D[\Gamma]) =1$, thus $D[\Gamma]_N$ is a
one-dimensional weakly Krull domain.
\end{proof}

Let $K$ be a field and $G$ be an additive torsion-free abelian group.
Then $K[G]$ is a Krull domain if and only if $G$ satisfies the ascending chain condition on
its cyclic subgroups \cite[Theorem 1]{c81}. Hence, the following result recovers the result by Fadinger and Windisch
\cite[Theorem 3.7]{fw21}.

\begin{corollary} \label{coro3}
Let $D$ be an integral domain with quotient field $K$, $\Gamma$ be a torsion-free monoid with quotient group $G$, and assume that
$K[G]$ is a weakly Krull domain. Then $D[\Gamma]$ is a weakly Krull domain if and only if
$D$ is a weakly Krull domain with ht$(P[\Gamma]) = 1$
for all $P \in t$-Max$(D)$ and $\Gamma$ is a weakly Krull monoid with ht$(D[S]) = 1$
for all $S \in t$-Max$(\Gamma)$.
\end{corollary}

\begin{proof}
By Lemma \ref{lemma1}(2), $D[\Gamma]_{N(H)}$ is one-dimensional if and only if
ht$(P[\Gamma]) = 1$ for all $P \in t$-Max$(D)$ and ht$(D[S]) = 1$
for all $S \in t$-Max$(\Gamma)$. Thus, the result follows directly from Lemma \ref{lemma1}(2), Proposition \ref{prop3}
and the fact that $D[\Gamma] = D[\Gamma]_{N(H)} \cap K[G]$.
\end{proof}

\begin{corollary} \cite[Proposition 4.11]{ahz93}
Let $D[X]$ be the polynomial ring over an integral domain $D$.
Then $D[X]$ is a weakly Krull domain if and only if $D$ is a weakly Krull UMT-domain.
\end{corollary}

\begin{proof}
Let $K$ be the quotient field of $D$. We may assume that $D$ is not a field, since if $D=K$ the statement is trivial.
Let $\Gamma = \{0, 1, 2, \dots \}$. Then $\Gamma$ is a torsion-free monoid under addition,
$D[X] = D[\Gamma]$, and $S: = \Gamma \setminus \{0\}$ is the unique nonempty prime ideal of $\Gamma$.
Furthermore, $\Gamma$ is a weakly Krull monoid, ht$(D[S])=1$, and $K[\langle \Gamma \rangle]$ is a UFD.
Hence, by Corollary \ref{coro3}, $D[\Gamma]$ is a weakly Krull domain if and only if
$D$ is a weakly Krull domain with ht$(P[\Gamma])=1$ for all $P \in t$-Max$(D)$,
if and only if $D$ is a weakly Krull UMT-domain (because $t$-dim$(D)=1$).
\end{proof}

\section{Weakly Krull semigroup rings}

In this section, we completely characterize when $D[\Gamma]$ is a weakly Krull domain.
Let $H$ be a monoid and $S$ be the set of non-invertible elements.
As in \cite[Theorem 15.4]{h98}, we say that $H$ is primary
if $S$ is the only non-empty prime ideal of $H$.

\begin{lemma} \label{lemma2}
Let $\Gamma$ be a primary monoid with quotient group $G$,
$S$ be the maximal ideal of $\Gamma$, and $\bar{\Gamma}$ be
the integral closure of $\Gamma$ in $G$. Let $K$ be a field,
and assume that $K[G]$ is a weakly Krull domain. Then ht$(K[S]) =1$ if and only if
$\bar{\Gamma}$ is a valuation monoid.
\end{lemma}

\begin{proof}
Note that $\bar{\Gamma} = \{\alpha \in G \mid n\alpha \in \Gamma$ for some integer $n \geq 1\}$,
so $\bar{\Gamma}$ is a primary monoid.
Hence, if we let $\bar{S}$ be the set of nonunits of $\bar{\Gamma}$,
then $\bar{S}$ is the unique non-empty prime ideal of $\bar{\Gamma}$ and $\bar{S} \cap \Gamma = S$.

\vspace{.2cm}
\noindent
{\Large {\bf Claim.}} ht$(K[S]) =$ ht$(K[\bar{S}])$. (Proof. Let $Q_0 \subsetneq Q_1 \subsetneq \cdots \subsetneq Q_n = K[\bar{S}]$ be a chain
of prime ideals of $K[\bar{\Gamma}]$. Then, since $K[\bar{\Gamma}]$ is integral over $K[\Gamma]$,
$$Q_0 \cap K[\Gamma] \subsetneq Q_1 \cap K[\Gamma] \subsetneq \cdots \subsetneq Q_n \cap K[\Gamma] = K[S]$$
is a chain of prime ideals of $K[\Gamma]$ \cite[Theorem 44]{kap}.
Hence ht$(K[S]) \geq$ ht$(K[\bar{S}])$. Next, let $P_0 \subsetneq P_1 \subsetneq \cdots \subsetneq P_m = K[S]$ be a chain
of prime ideals of $K[\Gamma]$. Then there is a chain $M_0 \subsetneq M_1 \subsetneq \cdots \subsetneq M_m$ of
 prime ideals of $K[\bar{\Gamma}]$ such that $M_i \cap K[\Gamma] = P_i$ for $i=0,1, \dots , m$ \cite[Theorem 44]{kap}.
 Note that $M_m \cap K[\Gamma] = K[S]$, so $K[\bar{S}] \subseteq M_m$. Note also that
 $K[\bar{S}] \cap K[\Gamma] = K[S]$. Thus, $M_m = K[\bar{S}]$ \cite[Theorem 44]{kap}, so
 ht$(K[\bar{S}]) \geq$ ht$(K[S])$. Hence, ht$(K[\bar{S}]) =$ ht$(K[S])$.)

\vspace{.2cm}

 Now, by Claim, we may assume that $\Gamma$ is integrally closed.

\vspace{.2cm}
 $(\Rightarrow)$ Assume to the contrary that $\Gamma$ is not a valuation monoid. Then
 there are $a, b \in \Gamma$ such that neither $a$ divides $b$ nor $b$ divides $a$ in $\Gamma$.
 Now, let $f = X^a +X^b \in K[\Gamma]$. Since $K[G]$ is a weakly Krull domain,
 $fK[G] = Q_1 \cap \dots \cap Q_k$ for some primary
 ideals $Q_i$ of $K[G]$ with ht$(\sqrt{Q_i}) = 1$ \cite[Theorem 3.1]{amz92}. Hence,
 $$fK[G] \cap K[\Gamma] = \bigcap_{i=1}^k (Q_i \cap K[\Gamma])$$
 and each $Q_i \cap K[\Gamma]$ is a primary ideal.

 Now, assume $fK[G] \cap K[\Gamma] \nsubseteq K[S]$.
Then there is an element $g \in K[G]$ such that $fg \in K[\Gamma] \setminus K[S]$. Hence,
 $C(fg) \nsubseteq K[S]$, and since $(C(f)C(g))_v = C(fg)_v$ \cite[Corollary 3.9]{aa82}, we have
 $C(f)C(g) \nsubseteq K[S]$. Note that $K$ is a field, so
 $C(f) = (X^a, X^b)K[\Gamma]$ and $C(g) = (X^{\alpha_1}, \dots , X^{\alpha_l})K[\Gamma]$
 for some $\alpha_i \in G$, whence either $X^{a+ \alpha_i} = X^aX^{\alpha_i} \not\in K[S]$ for some $i$ or
 $X^{b+ \alpha_j} = X^bX^{\alpha_j} \not\in K[S]$ for some $j$.
 We may assume that $X^{a+ \alpha_i} \not\in K[S]$. Then $a+\alpha_i \in \Gamma \setminus S$, so
 $a+\alpha_i$ is a unit of $\Gamma$, whence $a+\alpha_i + \beta =0$ for some $\beta \in \Gamma$.
 Thus, $b = a+ (b+ \alpha_i) + \beta$ and $(b+ \alpha_i) + \beta \in \Gamma$, which means that
 $a$ divides $b$ in $\Gamma$, a contradiction.

 Hence, $fK[G] \cap K[\Gamma] \subseteq K[S]$,
 and thus $Q_i \cap K[\Gamma] \subseteq K[S]$ for some $i$ with $1 \leq i \leq k$. Note that ht$(K[S])=1$ by assumption,
 so $K[S] = \sqrt{Q_i \cap K[\Gamma]}$, which implies $(K[S])K[G] \subsetneq K[G]$, a contradiction.
 Thus, $\Gamma$ is a valuation monoid.

 $(\Leftarrow)$ Assume that $(0) \neq Q \subsetneq K[S]$ is a chain of prime ideals of $K[\Gamma]$.
 Then, for $0 \neq f \in Q$, there is an $\alpha \in \Gamma$ such that
 $f = X^{\alpha}g$ for some $g \in K[\Gamma]$ with $C(g) = K[\Gamma]$.
 Hence, $g \not\in K[S]$, so $g \not\in Q$, and thus $X^{\alpha} \in Q$.
 But, in this case, if $S_1 = \{\alpha \in \Gamma \mid X^{\alpha} \in Q\}$, then
 $S_1$ is a non-empty prime ideal of $\Gamma$ and $K[S_1] \subseteq Q \subsetneq K[S]$. Thus,
 $S_1 \subsetneq S$, a contradiction. Therefore, ht$(K[S])=1$.
\end{proof}

\begin{lemma} \label{lemma3}
Let $G$ be a torsion-free abelian group, $D$ be a one-dimensional
local domain with maximal ideal $P$, $K$ be the quotient field of $D$,
and $\bar{D}$ be the integral closure of $D$ in $K$.
Assume that $K[G]$ is a weakly Krull domain. Then ht$(P[G]) =1$ if and only if
$\bar{D}$ is a Pr\"ufer domain.
\end{lemma}

\begin{proof}
The proof of this lemma is similar to that of Lemma \ref{lemma2}, but we give the proof
for the convenience of the reader.

{\Large {\bf Claim.}} ht$(P[G]) =$ max$\{$ht$(Q[G]) \mid Q$ is a maximal ideal of $\bar{D}\}$.
(Proof. Let $Q$ be a maximal ideal of $\bar{D}$, and let
$Q_0 \subsetneq Q_1 \subsetneq \cdots \subsetneq Q_n = Q[G]$ be a chain
of prime ideals of $\bar{D}[G]$. Then, since $\bar{D}[G]$ is integral over $D[G]$,
$$Q_0 \cap D[G] \subsetneq Q_1 \cap D[G] \subsetneq \cdots \subsetneq Q_n \cap D[G] = P[G]$$
is a chain of prime ideals of $D[\Gamma]$.
Hence ht$(P[G]) \geq$ ht$(Q[G])$. Next, let $P_0 \subsetneq P_1 \subsetneq \cdots \subsetneq P_m = P[G]$ be a chain
of prime ideals of $D[G]$. Then there is a chain $M_0 \subsetneq M_1 \subsetneq \cdots \subsetneq M_m$ of
 prime ideals of $\bar{D}[G]$ such that $M_i \cap D[G] = P_i$ for $i=0,1, \dots , m$.
 Note that $M_m \cap D[G] = P[G]$, so if we let $Q' = M_m \cap \bar{D}$, then $Q'$ is a maximal ideal of $\bar{D}$
 and $Q'[G] \subseteq M_n$. Note also that $Q'[G] \cap D[G] = P[G]$, thus $Q'[G] = M_n$ \cite[Theorem 44]{kap}.
Hence, ht$(P[G]) \leq$ ht$(Q'[G])$.)

\vspace{.2cm}
 Now, by Claim, we may assume that $D$ is an integrally closed one-dimensional domain
 (which need not be local). Then it suffices to show that
 ht$(P[G])=1$ for all maximal ideals $P$ of $D$ if and only if $D$ is a Pr\"ufer domain.

 $(\Rightarrow)$ Assume to the contrary that $D$ is not a Pr\"ufer domain. Then
 there are $a, b \in D \setminus \{0\}$ such that $(a,b)$ is not an invertible ideal of $D$. Hence
 $(a,b)(a,b)^{-1} \subseteq Q$ for some maximal ideal $Q$ of $D$.
 For $0 \neq \alpha \in G$, let $f = a + bX^{\alpha}$.
 Then, by \cite[Corollary 3.9]{aa82}, $$fK[G] \cap D[G] = fc(f)^{-1}[G],$$
 hence $fK[G] \cap D[G] \subsetneq Q[G]$.
 Since $K[G]$ is a weakly Krull domain, $fK[G]$ has a primary decomposition whose
 associated prime ideals have height-one \cite[Theorem 3.1]{amz92}, say, $fK[G] = Q_1 \cap \dots \cap Q_k$. Then
 $$fK[G] \cap D[G] = \bigcap_{i=1}^k (Q_i \cap D[G])$$
 and each $Q_i \cap D[G]$ is a primary ideal. Moreover, at least one of the $Q_i \cap D[G]$
 is contained in $Q[G]$, so ht$(Q[G]) \geq 2$, a contradiction. Thus,
 $D$ is a Pr\"ufer domain.

 $(\Leftarrow)$  Let $Q$ be a maximal ideal of $D$. Then ht$(Q[G]) =$ ht$(QD_Q[G])$
 and $D_Q$ is a valuation domain. Hence, we may assume that $D$ is a valuation domain. Then
 it is easy to see that ht$(Q[G]) = 1$ as in the proof $(\Leftarrow)$ of Lemma \ref{lemma2}.
\end{proof}

\begin{theorem} \label{theorem7}
Let $D$ be an integral domain with quotient field $K$ and $\Gamma$ be a torsion-free monoid with
quotient group $G$.
Then $D[\Gamma]$ is a weakly Krull domain if and only if $D$ is a weakly Krull UMT-domain,
$\Gamma$ is a weakly Krull UMT-monoid, and $K[G]$ is a weakly Krull domain.
\end{theorem}

\begin{proof}
$(\Rightarrow)$ By Proposition \ref{prop3},
$D[\Gamma]_{N(H)}$ is a one-dimensional weakly Krull domain and $K[G]$ is a weakly Krull domain.
Also, by Corollary \ref{coro3}, $D$ and $\Gamma$ are weakly Krull. Now, if $P \in t$-Max$(D)$,
then $1 =$ ht$(P) =$ ht$(P[\Gamma]) =$ ht$(P[G]) =$ ht$(PD_P[G])$
(see Corollary \ref{coro3} for the second equality), thus $\bar{D}_P$ is a Pr\"ufer domain
by Lemma \ref{lemma3}. Thus, $D$ is a UMT-domain.
Next, if $S \in t$-Max$(\Gamma)$, then ht$(K[S+\Gamma_S])=$ ht$(K[S]) =$ ht$(S)=1$ by Corollary \ref{coro3},
whence $\bar{\Gamma}_S$ is a valuation monoid by Lemma \ref{lemma2}. Thus, $\Gamma$ is a UMT-monoid.

$(\Leftarrow)$ Note that ht$(P[\Gamma]) =$ ht$(P[G]) =$ ht$(PD_P[G]) =1$ for all $P \in t$-Max$(D)$
by Lemma \ref{lemma3}
and ht$(D[S]) =$ ht$(K[S]) =$ ht$(K[S+\Gamma_S]) =1$ for all $S \in t$-Max$(\Gamma)$ by Lemma \ref{lemma2}.
Recall that
$$\text{Max}(D[\Gamma]_{N(H)}) = \{P[\Gamma]_{N(H)} \mid P \in t\text{-Max}(D)\} \cup
\{D[S]_{N(H)} \mid S \in t\text{-Max}(\Gamma)\}.$$ Thus, $D[\Gamma]_{N(H)}$ is a
one-dimensional weakly Krull domain. Therefore, $D[\Gamma]$ is a weakly Krull domain by Proposition \ref{prop3}.
\end{proof}

The following lemma is from \cite[Corollaries 3.2 and 4.3]{co19}.

\begin{lemma} \label{lemma8}
Let $D$ be an integral domain with quotient field $K$ and char$(D)=0$ (resp., char$(D) = p >0$). Then
$K[G]$ is a weakly Krull domain if and only if $G$ is of type $(0,0,0, \dots )$
(resp., of type $(0,0,0, \dots )$ except $p$).
\end{lemma}

By Theorem \ref{theorem7} and Lemma \ref{lemma8}, we have the following two corollaries
which are complete characterizations of semigroup rings $D[\Gamma]$ that are weakly Krull domains.

\begin{corollary} \label{coro8}
Let $D$ be an integral domain, $\Gamma$ be a torsion-free monoid with quotient group $G$, and assume char$(D)=0$.
Then $D[\Gamma]$ is a weakly Krull domain if and only if $D$ is a weakly Krull UMT-domain,
$\Gamma$ is a weakly Krull UMT-monoid, and $G$ is of type $(0,0,0, \dots )$.
\end{corollary}

\begin{corollary} \label{coro9}
Let $D$ be an integral domain, $\Gamma$ be a torsion-free monoid with quotient group $G$, and assume char$(D)=p > 0$.
Then $D[\Gamma]$ is a weakly Krull domain if and only if $D$ is a weakly Krull UMT-domain,
$\Gamma$ is a weakly Krull UMT-monoid, and $G$ is of type $(0,0,0, \dots )$ except $p$.
\end{corollary}

Let $\mathbb{N}_0$ be the additive monoid of nonnegative integers under the
usual addition. Then $\mathbb{N}_0$ is a torsion-free monoid with quotient group $\mathbb{Z}$.
A numerical monoid $\Gamma$ is a submonoid of $\mathbb{N}_0$ such that $0 \in \Gamma$ and $\mathbb{N} \setminus \Gamma$ is finite. Hence,
$\Gamma$ is a torsion-free monoid with quotient group $\mathbb{Z}$.

\begin{corollary}\label{coro10} \cite[Theorem 1.3]{l12}
Let $D$ be an integral domain and $\Gamma$ be a numerical monoid with
$\Gamma \subseteq \mathbb{N}_0$. Then $D[\Gamma]$ is a weakly Krull domain
if and only if $D$ is a weakly Krull UMT-domain.
\end{corollary}

\begin{proof}
Clearly, $\mathbb{N}_0$ is the integral closure of $\Gamma$ in $\mathbb{Z}$
and $\mathbb{N}_0$ is a valuation monoid. Moreover, $\Gamma \setminus \{0\}$ is the unique
nonempty prime ideal of $\Gamma$, so $\Gamma$ is a weakly Krull UMT-monoid.
Note that if $K$ is the quotient field of $D$, then
$K[\mathbb{Z}]$ is a UFD, and hence a weakly Krull domain.
Thus, the proof is completed by Theorem \ref{theorem7}.
\end{proof}

A generalized Krull domain $D$ is a weakly Krull domain such that $D_P$ is a valuation domain
for all $P \in t$-Max$(D)$. The next result was proved by Matsuda (\cite[Proposition 10.7]{m77} for the case of
char$(D)=0$ and \cite[Theorems 1.5 and 4.3]{m82} for the case of char$(D)=p>0$).

\begin{corollary}
Let $D$ be an integral domain, $\Gamma$ be a torsion-free monoid with quotient group $G$, and assume char$(D)=0$ (resp., char$(D) = p >0$).
Then $D[\Gamma]$ is a generalized Krull domain if and only if $D$ is a generalized Krull domain,
$\Gamma$ is a generalized Krull monoid, and $G$ is of type $(0,0,0, \dots )$
(resp., type $(0,0,0, \dots )$ except $p$).
\end{corollary}

\begin{proof}
Let $P$ be a prime ideal of $D$. Then $D_P$ is a valuation domain if and only if
$D[\Gamma]_{P[\Gamma]}$ is a valuation domain. Also, if $S$ is a prime ideal of $\Gamma$,
then $\Gamma_S$ is a valuation monoid if and only if
$D[\Gamma]_{D[S]}$ is a valuation domain. Thus, the result follows directly from Corollaries \ref{coro8} and \ref{coro9}.
\end{proof}

It is well known that a domain $D$ is a WFD if and only if it is weakly Krull with $Cl_t(D) = \{0\}$ 
\cite[Theorem]{az90}.  
We next use the main result of this section to recover Chang and Oh's result \cite{co19}
which completely characterizes
when $D[\Gamma]$ is a WFD.

\begin{corollary} \label{wfd}
\cite[Theorems 3.4 and 4.5]{co19}
Let $D$ be an integral domain, $\Gamma$ be a torsion-free monoid with quotient group $G$, and assume char$(D)=0$
(resp., char$(D) = p >0$).
Then $D[\Gamma]$ is a WFD if and only if $D$ is a weakly factorial GCD-domain,
$\Gamma$ is a weakly factorial GCD-monoid, and $G$ is of type $(0,0,0, \dots )$
(resp., $(0,0,0, \dots )$ except $p$).
\end{corollary}

\begin{proof}
This follows directly from Proposition \ref{prop6}, Corollaries \ref{coro8}, \ref{coro9}, and the following
observations: (i) If $D$ and $\Gamma$ are integrally closed, then $Cl_t(D[\Gamma]) = Cl_t(D) \times Cl_t(\Gamma)$
\cite[Lemma 2.1 and Corollary 2.11]{eik02}, (ii) if $Cl_t(D[\Gamma]) = \{0\}$,
then $D$ and $\Gamma$ are integrally closed and $Cl_t(D) = Cl_t(\Gamma) = \{0\}$ by \cite[Theorems 2.6 and 2.7]{eik02},
(iii) $D[\Gamma]$ is a P$v$MD if and only if $D$ is a P$v$MD and $\Gamma$ is a P$v$MM \cite[Proposition 6.5]{aa82},
(iv) $D$ is a P$v$MD if and only if $D$ is an integrally closed UMT-domain \cite[Proposition 3.2]{hz89},
(v) $D$ is a GCD-domain if and only if $D$ is a P$v$MD with $Cl_t(D) = \{0\}$ \cite[Proposition 2]{b82},
(vi) $\Gamma$ is a GCD-monoid if and only if $\Gamma$ is a P$v$MM with $Cl_t(\Gamma) = \{0\}$
\cite[Theorem 11.5]{h98}, and (vii)  $\Gamma$ is a weakly factorial monoid if and only if
$\Gamma$ is a weakly Krull monoid with $Cl_t(\Gamma)= \{0\}$ \cite[p. 258]{h98}.
\end{proof}

\section{Arithmetical applications of the main result}

Building on the algebraic results of the previous section and on a recent work on the distribution of prime divisors in the class groups of affine semigroup rings \cite{Affine}, we will study the factorization theory of weakly Krull semigroup rings. What is known up to now concerning factorizations in weakly Krull domains are on the one hand a few very general results lacking examples and on the other hand very concrete examples lacking generality. This is mainly due to the fact
 that the class group and the distribution of the prime divisors play a key role in the investigation of the factorization behaviour of a weakly Krull domain. But determining class groups and prime divisors in the classes is in general very hard. Nevertheless, the structure of sets of lengths of one-dimensional local Mori domains (equivalently local weakly Krull Mori domains) is given in \cite{local}.

We are the first who give a fairly broad but sufficiently concrete class of non-local weakly Krull domains
that are not Krull domains, namely certain affine semigroup rings, where we can understand the arithmetic.
For example, up to now the knowledge of domains having full system of sets of lengths was restricted to a class of certain Krull domains
(see \cite{Kainrath}) and a class of integer-valued polynomial rings (see \cite{Frisch1, Frisch2}).
 Our results show that there is also a class of weakly Krull domains, which are not Krull but have full system of sets of lengths.

We recall some concepts from factorization theory (for details see \cite[Chapter 1]{GHK}). Let $H$ be a monoid and let  $\mathcal A(H)$ be its set of atoms. We denote by $H_{red}=H/H^\times$ the associated reduced monoid.
Consider the free abelian monoid $\mathsf Z(H)=\mathcal F(\mathcal A(H_{red}))$ with the epimorphism $\pi:\mathsf Z(H)\to H_{red}$ via $\pi(uH^\times)=uH^\times$ for all $u\in \mathcal A(H)$.
For $a\in H$,
\begin{itemize}
\item $\mathsf Z(a)=\pi^{-1}(\{aH^\times\})$ is the \textit{set of factorizations} of $a$, and

\item  $\mathsf L(a)=\{|z|\colon z\in\mathsf Z(a)\}$ is  the \textit{set of lengths}  of $a$, where $|z|=m$ if $z=u_1\cdots u_m$ for $u_i\in \mathcal A(H_{red})$.
\end{itemize}
Then $H$ is said to be \textit{atomic} if $\mathsf L(a)$ is non-empty for all $a\in H$
and is said to be a BF-\textit{monoid} if  $H$ is atomic and $\mathsf L(a)$ is finite for all $a\in H$.
Note that if $R$ is a Mori domain, then the multiplicative monoid $R\setminus\{0\}$ is always a BF-monoid.
We consider the system
$\mathcal L(H)=\{\mathsf L(a)\colon a\in H\}$ of all sets of lengths of $H$. For convenience,
we denote $\mathcal L(R\setminus\{0\})$ by $\mathcal L(R)$ for an integral domain $R$. A system of sets of lengths of a BF-monoid is called \textit{full} if it equals $\{\{0\},\{1\}\}\cup \{L\subseteq \N_{\geq 2}\mid L\text{ finite non-empty}\}$.  The \textit{set of distances} of $H$ is $\Delta(H)=\bigcup_{L\in\mathcal L(H)}\Delta(L)$, where $\Delta(L)=\{d\in\N\colon \text{there is } l\in L \text{ such that } L\cap \{l,l+1,\ldots,l+d\}=\{l,l+d\}\}$. For  $k\in \N$, we define $\mathcal U_k(H)=\bigcup_{k\in L\in\mathcal L(H)} L$.

Let $R$ be a weakly Krull Mori domain with non-zero conductor $\mathfrak{f}_R = (R: \widehat{R}) \neq (0)$. 
Then $H = R \setminus \{0\}$ is a weakly Krull Mori monoid with non-empty conductor $\mathfrak{f}_H = (H:\widehat{H}) \neq \emptyset$. 
It follows by \cite[Proposition 2.4.5.1]{GHK} that the canonical map 
$H \to \mathcal{I}_v^*(H) \cong \coprod_{\mathfrak{p} \in X^1(H)} (H_\mathfrak{p})_{red} \cong \mathcal{F}(P) \times D_1 \times \ldots \times D_n$ 
is a divisor homomorphism, where $P = \{ \mathfrak{p} \in X^1(H) \mid \mathfrak{p} \not\supseteq \mathfrak{f}_H\}$, $\mathcal{F}(P)$ 
is the free abelian monoid with basis $P$, and $D_i = (H_{\mathfrak{q}_i})_{red}$ for $i \in \{1,\ldots,n\}$ 
and $X^1(H) \setminus P = \{\mathfrak{q}_1,\ldots, \mathfrak{q}_n\}$. The existence of the first isomorphism is proven in \cite[Proposition 5.3.4]{GKR} and the second isomorphism as well as the finiteness of $X^1(H) \setminus P$ follows from \cite[Theorem 2.6.5]{GHK}.
Let $G = \mathcal{C}_v(H) \cong \mathcal{C}_v(R)$ be the divisor-class group of $H$ 
and let $G_0$ be the set of all classes containing prime divisors, that is, prime ideals 
$\mathfrak{p} \in P$. Let $T = D_1 \times \cdots \times D_n$ (so $T$ is a reduced monoid, i.e. $T^\times$ is trivial) 
and let $\iota: T \to G$ be the canonical map induced by the isomorphisms from above and the projection $\mathcal{I}_v^*(H) \to G$.
The \textit{$T$-block monoid} over $G_0$ defined by $\iota$ is
\begin{align*}
B = \mathcal{B}(G_0,T,\iota) = \{(g_1 \cdots g_k,t) \in \mathcal{F}(G_0) \times T \mid g_1 + \ldots + g_k + \iota(t) = 0\}.
\end{align*}
Then the monoid $\mathcal{B}(G_0) = \{g_1\cdots g_k \in \mathcal{F}(G_0) \mid g_1+ \ldots + g_k = 0 \}$ is a divisor-closed submonoid of $B$. By \cite[Lemma 4.3]{GKR}, there exists a transfer homomorphism $\beta: H \to B$. Thus, $\mathcal{L}(R) = \mathcal{L}(H) = \mathcal{L}(B)$ by \cite[Proposition 3.2.3.5]{GHK}. Moreover, it follows that $\mathcal{L}(\mathcal{B}(G_0)) \subseteq \mathcal{L}(R)$ by the previous equality and $\mathcal{B}(G_0) \subseteq B$ being divisor-closed. It is easy to see that $\mathcal B(G_1)\subseteq \mathcal B(G_0)$ is a divisor closed submonoid for every subset $G_1\subseteq G_0$, whence $\mathcal L(\mathcal B(G_1))\subseteq \mathcal L(\mathcal B(G_0))$.

The notion of Hilbertian fields is a classical one whose origin lies in Galois theory and is to be found in \cite{Fried}. For our purpose, we need a generalization of it.

\begin{definition}
A field $K$ is called \textit{pseudo-Hilbertian} if, for all $n\in \N_0$ and for all $a_0,\ldots ,a_n\in K$ with $a_0\neq 0$, there exists an irreducible polynomial in $K[X]$ whose coefficient at the monomial $X^i$ equals $a_i$ for all $i\in\{0,\ldots,n\}$.
\end{definition}

Note that every Hilbertian field is an infinite pseudo-Hilbertian field. In particular,
algebraic function fields over an arbitrary field and algebraic number fields are pseudo-Hilbertian.
See \cite{Fried} for more on Hilbertian fields. Moreover, finite fields are pseudo-Hilbertian \cite{Pollack}.

For the following theorem, note that if $D$ is noetherian and $\Gamma$ is a numerical monoid, then $D[\Gamma]$ is noetherian and hence Mori.

\begin{theorem}\label{main:1}
Let $D$ be a weakly Krull UMT-domain with non-zero conductor $\mathfrak{f}_D = (D:\widehat{D}) \neq (0)$
and infinite pseudo-Hilbertian quotient field $K$.
Let $\Gamma\neq \N_0$ be a numerical monoid and suppose that $D[\Gamma]$ is a Mori domain. Then $\mathcal L(D[\Gamma])$ is full.
\end{theorem}

\begin{proof}
By Corollary \ref{coro10} and assumption, $D[\Gamma]$ is a weakly Krull Mori domain.
Then $$\mathfrak f_{D[\Gamma]}=(D[\Gamma]:\widehat{D[\Gamma]})\neq (0)$$ by  \cite[Lemma 3.1]{Affine},
whence we are in the situation that we explained at the beginning of this section.
The class group of $D[\Gamma]$ is of the form $\mathcal C_v(D[\Gamma])\cong \mathcal C_v(D[X])\oplus \Pic(K[\Gamma])$ \cite[Theorem 5]{Chang_numerical}.
Since $K$ is infinite and pseudo-Hilbertian, $K[\Gamma]$ has infinitely many prime divisors in every class by \cite[Theorem 1]{Affine}
and  $\Pic(K[\Gamma])$ is infinite by \cite[Propositions 3.4 \& 3.7]{Affine}.
So if $G_0\subseteq \mathcal C_v(D[\Gamma])$ denotes the set of classes containing prime divisors,
then $G_0$ contains the infinite abelian group $\Pic(K[\Gamma])$.
Hence $\mathcal B(\Pic(K[\Gamma]))\subseteq \mathcal B(G_0)\subseteq \mathcal B(G_0,T,\iota)$.
Therefore $\mathcal L(\mathcal B(\Pic(K[\Gamma]))\subseteq \mathcal L(D[\Gamma])$
and the statement follows by Kainrath's Theorem \cite{Kainrath}.
\end{proof}

To give two important special cases of Theorem \ref{main:1}, we apply it to orders in algebraic number fields and to polynomial rings.

\begin{corollary}
Let $\Gamma\neq \N_0$ be a numerical monoid.
\begin{enumerate}
\item If $D$ is an order in an algebraic number field, then $\mathcal L(D[\Gamma])$ is full.
\item If $D$ is a noetherian weakly Krull UMT-domain with non-zero conductor, then $\mathcal L(D[X][\Gamma])$ is full.
\end{enumerate}
\end{corollary}
\begin{proof}
(1) If $D$ is an order in an algebraic number field, then $D$ is a noetherian weakly Krull UMT-domain
(the integral closures of the localizations at maximal $t$-ideals are one-dimensional Krull by Mori-Nagata Theorem,
whence Pr\"ufer) with non-zero conductor and infinite pseudo-Hilbertian quotient field. Now the assertion follows from Theorem \ref{main:1}.

(2) If $D$ is a noetherian weakly Krull UMT-domain with non-zero conductor, then $D[X]$ is noetherian and a UMT-domain \cite[Theorem 2.4]{fgh98}. It follows from Corollary \ref{coro10} that $D[X]$ is weakly Krull and from \cite[Lemma 3.1]{Affine} that the conductor $\mathfrak f_{D[X]}$ is non-zero. Since the quotient field $K(X)$ of $D[X]$ is infinite and pseudo-Hilbertian, we can apply Theorem \ref{main:1}.
\end{proof}

Recall that a monoid $H$ with quotient group $\langle H\rangle$ is said to be \textit{seminormal}
if for all $x\in \langle H\rangle$ we have that $x^2, x^3\in H$ implies $x\in H$. For the next theorem,
 note that seminormal affine monoids are characterized in terms of their geometry,
  e.g. see \cite[Proposition 2.42]{Bruns}. Also, for seminormal weakly Krull affine monoids,
  either statement (1) or statements (2) and (3) of the next theorem are true always. The theorem is particularly interesting because even in the case of orders $\mathcal O$ in algebraic number fields, $\min(\Delta(\mathcal O))>1$ can occur.

\begin{theorem}\label{main:2}
Let $K$ be a field, $\Gamma$ be a weakly Krull affine monoid
that is not a numerical monoid, and assume that the root closure $\bar{\Gamma}$ is factorial. Then either $K[\Gamma]$ is half-factorial or $\min(\Delta(K[\Gamma]))=1$. Moreover, the following hold true.
\begin{enumerate}
\item If $\mathcal C_v(K[\Gamma])$ is infinite, then $\mathcal L(K[\Gamma])$ is full.
\item If $\mathcal C_v(K[\Gamma])$ is finite, then $K[\Gamma]$ satisfies the Structure Theorem for Sets of Lengths (see \cite[Definition 4.7.1]{GHK}).
\item If $\mathcal C_v(K[\Gamma])$ is finite and $\Gamma$ is seminormal, then
both $\Delta(K[\Gamma])$ and $\mathcal U_k(K[\Gamma])$ are finite intervals for all $k\geq 2$.
\end{enumerate}
\end{theorem}

\begin{proof}
Note that every weakly Krull affine monoid is a UMT-monoid,
because its localizations at maximal $t$-ideals are finitely generated primary monoids,
whence their integral closures are primary Krull monoids, that is, discrete rank one valuation monoids.
Thus, $K[\Gamma]$ is a weakly Krull domain by Corollary \ref{coro10}.
Moreover, $K[\Gamma]$ is noetherian, $K[\bar{\Gamma}]=K[\widehat{\Gamma}]=\widehat{K[\Gamma]}$ and there is a one-to-one correspondence of
height-one prime ideals $X^1(K[\bar{\Gamma}])\to X^1(K[\Gamma])$ given by $P\mapsto P\cap K[\Gamma]$ using a combination of Lemma \ref{lemma1} and \cite[Proposition 2.7]{Rein}.
Since every finitely generated monoid always has a non-empty conductor \cite[Theorem 2.7.13]{GHK},
it follows from \cite[Lemma 3.1]{Affine} that $K[\Gamma]$ has a non-zero conductor.
Thus, we are in the situation that we explained at the beginning of this section.
Moreover, by \cite[Theorem 2]{Affine}, there are infinitely many prime divisors in all classes of $\mathcal C_v(K[\Gamma])$.

The statement on the half-factoriality and the minimum of $\Delta(K[\Gamma])$ follows from \cite[Theorem 1.1]{GZ}.

(1)  Let $\mathcal C_v(K[\Gamma])$ be infinite. Then $\mathcal B(\mathcal C_v(K[\Gamma]))\subseteq \mathcal B(\mathcal C_v(K[\Gamma]), T, \iota)$. Therefore $\mathcal L(\mathcal B(\mathcal C_v(K[\Gamma])))\subseteq \mathcal L(K[\Gamma])$ and the statement follows by Kainrath's Theorem \cite{Kainrath}.

(2) This is immediate by \cite[Chapter 4.7]{GHK}.

(3) If $\Gamma$ is seminormal, then $K[\Gamma]$ is seminormal by \cite[Theorem 4.75]{Bruns}.
Thus $\mathcal U_k(K[\Gamma])$ (resp., $\Delta(K[\Gamma])$) is a finite interval for all $k\geq 2$ by \cite[Theorem 5.8.2 (a)]{GKR}
(resp., \cite[Theorem 1.1]{GZ}).
\end{proof}

\begin{remark}
{\em Let $R$ be a weakly Krull Mori domain. Then the monoid $H = \mathcal I^* (R )$ is a weakly Krull Mori monoid.
If $R$ has a nonzero conductor, then $H$ has a nonzero conductor; if $R$ is seminormal, then $H$ is seminormal;
if the $v$-class group $\mathcal C_v(R)$ of $R$ has (infinitely) many prime divisors in the classes,
then the same is true for $H$ (see \cite[Theorem 4.4 \& Corollary 4.7]{GeA}).
Thus, all the mentioned arithmetical properties for $R$ hold true for $H$ too.}
\end{remark}

We close this section with an application of Theorem \ref{main:2}(1). We first
need the following lemma.

\begin{lemma}\label{lemma:numerical}
Let $K$ be a field and $S_1, \ldots ,S_n$ be numerical monoids with $n>1$. Then
\[
\mathcal C_v(K[\bigoplus_{i=1}^n S_i])\cong \bigoplus_{i=1}^n \mathcal C_v(K(X_1,\ldots ,X_{n-1})[S_i]).
\]
\end{lemma}
\begin{proof}
Note that $\mathcal C_v(K[\Z^m])$ is trivial for all integers $m \geq 1$, so
it suffices to show that for all positive integers $m$ with $m\leq n$,
$$\mathcal C_v(K[\bigoplus_{i=1}^n S_i])\cong \mathcal C_v(K[\Z^m\oplus S_{m+1}\oplus\ldots\oplus S_n])\oplus\bigoplus_{i=1}^m \mathcal C_v(K(X_1,\ldots ,X_{n-1})[S_i]).$$
We prove the isomorphism by induction on $m \leq n$.
First if $m=1$, then
\begin{eqnarray*}
\mathcal C_v(K[\bigoplus_{i=1}^n S_i]) &\cong& \mathcal C_v(K[\bigoplus_{i=2}^n S_i][S_1])\\
 &\cong& \mathcal C_v(K[\bigoplus_{i=2}^n S_i][\Z])\oplus \mathcal C_v(K(X_1,\ldots , X_{n-1})[S_1])\\
 &\cong& \mathcal C_v(K[\Z\oplus\bigoplus_{i=2}^n S_i])\oplus \mathcal C_v(K(X_1,\ldots , X_{n-1})[S_1]),
 \end{eqnarray*}
 where the second isomorphism follows from \cite[Theorem 5]{Chang_numerical}.
 Now assume that $m>1$. Then by the induction hypothesis,
\begin{align*}
\mathcal C_v(K[\bigoplus_{i=1}^n S_i])&\cong \mathcal C_v(K[\Z^{m-1}\oplus S_{m}\oplus\ldots\oplus S_n])\oplus\bigoplus_{i=1}^{m-1} \mathcal C_v(K(X_1,\ldots ,X_{n-1})[S_i])\\
&\cong \mathcal C_v(K[\Z^{m-1}\oplus \bigoplus_{i=m+1}^n S_i][S_m])\oplus \bigoplus_{i=1}^{m-1} \mathcal C_v(K(X_1,\ldots ,X_{n-1})[S_i])\\
&\cong \mathcal C_v(K[\Z^{m-1}\oplus \bigoplus_{i=m+1}^n S_i][\Z])\oplus\bigoplus_{i=1}^m \mathcal C_v(K(X_1,\ldots ,X_{n-1})[S_i])\\
&\cong \mathcal C_v(K[\Z^m\oplus S_{m+1}\oplus\ldots\oplus S_n])\oplus\bigoplus_{i=1}^m \mathcal C_v(K(X_1,\ldots ,X_{n-1})[S_i]).
\end{align*}
Thus, the isomorphism holds for all positive integers $m$ with $m \leq n$.
\end{proof}

\begin{corollary}
Let $K$ be a field, $S_1, \ldots ,S_n$ be numerical monoids with $n>1$
such that $S_i\neq \N_0$ for at least one of the $S_i$, and $\Gamma=\bigoplus_{i=1}^n S_i$.
Then $\mathcal L(K[\Gamma])$ is full.
\end{corollary}
\begin{proof}
Clearly, $\Gamma$ is an affine monoid and by \cite[Proposition 5.8]{FZ} it is weakly Krull, so we only need to show that $K[\Gamma]$ has an infinite class group.
Then we can apply Theorem \ref{main:2}(1). But this follows immediately from Lemma \ref{lemma:numerical}
in combination with \cite[Propositions 3.4 and 3.7]{Affine}.
\end{proof}

\vspace{.2cm}
\noindent
\textbf{Acknowledgements.}
The first named author was supported  by the Incheon National University Research
Grant in 2021. The second and third authors' work was supported by the Austrian Science Fund FWF, Projects W1230 and P~30934.

\end{document}